\documentclass[12pt]{amsart}

\usepackage[utf8]{inputenc}
\usepackage[english]{babel}

\usepackage[letterpaper,top=2cm,bottom=2cm,left=3cm,right=3cm,marginparwidth=1.75cm]{geometry}

\usepackage{amsmath}
\usepackage{graphicx}
\usepackage[colorlinks=true, allcolors=blue]{hyperref}

\usepackage{amssymb} 
\usepackage{comment}
\usepackage{indentfirst}
\usepackage{csquotes}
\usepackage{theoremref}
\usepackage{hyperref}
\usepackage{geometry}

\hypersetup{hidelinks, colorlinks=true, allcolors=cyan, pdfstartview=Fit, breaklinks=true}

\usepackage{amsthm}

\theoremstyle{plain}
\newtheorem{theorem}{Theorem}
\numberwithin{theorem}{section}
\newtheorem{corollary}[theorem]{Corollary}
\newtheorem*{corollary*}{Corollary}
\newtheorem*{Example*}{Example}

\newtheorem{proposition}[theorem]{Proposition}

\theoremstyle{definition}

\newtheorem*{def*}{Definition}
\newtheorem*{theorem*}{Theorem}

\newtheorem*{problem}{Problem}

\newtheorem*{definition*}{Definition}

\theoremstyle{remark}
\newtheorem*{remark}{Remark}

\newcommand{\bracket}[1]{\left( #1 \right)}
\newcommand{\floor}[1]{\left\lfloor #1 \right\rfloor}
\newcommand{\modulo}[3]{#1\equiv#2\ \bracket{\mathrm{mod}\ #3}}

\numberwithin{equation}{section}

\title{\textbf{New congruences for 4,6-regular partitions modulo primes}}
\author{QI-YANG ZHENG}
\date{} 

\address{Department of Mathematics, Sun Yat-sen University(Zhuhai Campus), China}
\email{zhengqy29@mail2.sysu.edu.cn}

\begin{document}

\begin{abstract}
The main result of the paper is the existence of an infinitely many families of Ramanujan-type congruences for $b_4(n)$ and $b_6(n)$ modulo primes $m \geq 2$ and $m \geq 5$, respectively. We provide new examples of congruences for $b_4(n)$ and $b_6(n)$. Moreover, we find two infinite explicit infinite families of congruences for $b_4(n)$ modulo $3$.
\end{abstract}

\maketitle

~

\section{Introduction}

The number of partitions of $n$ in which no parts are multiples of $k$ is denoted by $b_k(n)$ and referred to as $k$-regular partitions.

We agree that $b_k(0)=1$ for convenience. Moreover, let $b_k(n)=0$ if $n\not\in\mathbb{Z}_{\geq0}$. The generating function for the $k$-regular partitions is as follows:

\begin{equation}
    \abovedisplayskip=1em
    \notag
    \sum_{n=0}^\infty b_k(n)q^n=\prod_{n=1}^\infty\frac{1-q^{kn}}{1-q^n}.
    \belowdisplayskip=1em
\end{equation}

~

In 1919, Ramanujan discovered three remarkable congruences for the unrestricted partition function $p(n)$, as shown below:
\begin{equation}
    \abovedisplayskip=1em
    \notag
    \begin{aligned}
        p(5n+4)&\equiv0\ (\mathrm{mod}\ 5),\\
        p(7n+5)&\equiv0\ (\mathrm{mod}\ 7),\\
        p(11n+6)&\equiv0\ (\mathrm{mod}\ 11).
    \end{aligned}
    \belowdisplayskip=1em
\end{equation}

We refer to such congruences as Ramanujan-type congruences. Lovejoy \cite{lovejoy2001divisibility} studied the distribution of $b_2(n)$ and proved the existence of Ramanujan-type congruences for $b_2(n)$ modulo every prime $m\geq5$. In a previous work, the author \cite{zheng2022distribution} demonstrated similar results for $b_3(n)$ and $b_5(n)$.

The primary result of this paper is as follows:

\begin{theorem}
\label{infinitely many Ramanujan-type congruences}

~

\begin{enumerate}
    \item For every prime $m\geq2$, there exist infinitely many Ramanujan-type congruences of $b_4(n)$ modulo $m$.
    \item For every prime $m\geq5$, there exist infinitely many Ramanujan-type congruences of $b_6(n)$ modulo $m$.
\end{enumerate}
\end{theorem}

\noindent
Specifically, we find two infinite explicit families of congruences for $b_4(n)$ modulo $3$.

\begin{theorem}
\label{congruences mod 3}
If $l$ is a prime with $\modulo{l}{13,17,19,23}{24}$, then
\begin{equation}
    \abovedisplayskip=1em
    \notag
    b_4\bracket{3l(ln+j)+\frac{l^{2}-1}{8}}\equiv0\ (\mathrm{mod}\ 3),
    \belowdisplayskip=1em
\end{equation}

\noindent
and
\begin{equation}
    \abovedisplayskip=1em
    \notag
    b_4\bracket{3l(ln+j)+\frac{9l^{2}-1}{8}}\equiv0\ (\mathrm{mod}\ 3),
    \belowdisplayskip=1em
\end{equation}

\noindent
for every $n$ and $1\leq j\leq l-1$.
\end{theorem}

\begin{remark}
    If $l$ is a prime with $\modulo{l}{13,17,19,23}{24}$, then
    \begin{equation}
        \abovedisplayskip=1em
        \notag
        b_4\bracket{3l(ln+j)+\frac{17l^{2}-1}{8}}\equiv0\ (\mathrm{mod}\ 3)
        \belowdisplayskip=1em
    \end{equation}

    \noindent
    may not always hold. For example, let $l=13$, $n=0$, and $j=1$, we have $\modulo{b_4(398)}13$.
\end{remark}

We will now present the results that have been established for these regular partitions up to the present moment. Notably, a surprising finding is that \cite{andrews2010arithmetic}
\begin{equation}
    \abovedisplayskip=1em
    \notag
    \sum_{n=0}^\infty b_4(9n+7)q^n=12\frac{(q^2;q^2)_\infty^4(q^3;q^3)_\infty^6(q^4;q^4)_\infty}{(q;q)_\infty^{11}},
    \belowdisplayskip=1em
\end{equation}

\noindent
where $(a;q)_\infty=\prod_{n=0}^\infty(1-aq^n)$. A surprising finding is that it provides Ramanujan-type congruences of $b_4(n)$ modulo $2$ and $3$.

In the case of $b_6(n)$, the only known instances pertain to congruences modulo $3$, as shown in \cite[Theorem 1.6]{ahmed2016new}. For instance, we have
\begin{equation}
    \abovedisplayskip=1em
    \notag
    \modulo{b_6(169n+48)}{0}{3}.
    \belowdisplayskip=1em
\end{equation}

\begin{remark}
    N. D. Baruah, one of the authors of \cite{ahmed2016new}, has notified me that the exponent of $p$ in the second term of the sum (1.7) in \cite{ahmed2016new} should be adjusted to $2\alpha+2$.
\end{remark}

Our findings yield new congruences for $b_4(n)$ and $b_6(n)$; for instance,
\begin{equation}
    \abovedisplayskip=1em
    \notag
    \begin{gathered}
        \modulo{b_4(507 n + 34)}{0}{3},\\
        \modulo{b_4(3272405 n + 2528)}{0}{5},\\
        \modulo{b_4(24978247 n + 11570)}{0}{7},\\
        \modulo{b_6(973182225 n + 2055)}{0}{5}.
    \end{gathered}
    \belowdisplayskip=1em
\end{equation}

\noindent
More examples will be provided in Section \ref{more example}.

\begin{remark}
    The existence of such congruences are guaranteed by the work of Treneer \cite{treneer2006congruences}. However, the method he provided is difficult to obtain such congruences. We present an explicit method for obtaining congruences.
\end{remark}

\section{Preliminaries on modular forms}

\noindent
First, we will introduce the $U$ operator. If $j$ is a positive integer,
\begin{equation}
    \abovedisplayskip=1em
    \notag
    \left( \sum_{n=0}^\infty a(n)q^n \right)\ |\ U(j):=\sum_{n=0}^\infty a(jn)q^n.
    \belowdisplayskip=1em
\end{equation}

\noindent
If $\sum_{n=0}^\infty a(n)q^n\in M_{k}(\Gamma_0(N),\chi)$ represents a modular form and $m$ is a prime,
\begin{equation}
    \abovedisplayskip=1em
    \notag
    \left( \sum_{n=0}^\infty a(n)q^n \right)\ |\ T(m):=\sum_{n=0}^\infty (a(mn)+\chi(m)m^{k-1}a(n/m))q^n,
    \belowdisplayskip=1em
\end{equation}

\noindent
where $a(n/m)=0$ if $m\nmid n$. The $T(m)$ operator corresponds to the standard Hecke operator. It's worth recalling that Dedekind's eta function is defined by
\begin{equation}
    \abovedisplayskip=1em
    \notag
    \eta(z)=q^\frac{1}{24}\prod_{n=1}^\infty (1-q^n),
    \belowdisplayskip=1em
\end{equation}

\noindent
where $q=e^{2\pi iz}$.

~

If $m$ is a prime, we denote by $M_{k}(\Gamma_0(N),\chi)_m$ (respectively, $S_{k}(\Gamma_0(N),\chi)_m$) the $\mathbb{F}_m$-vector space obtained by reducing the $q$-expansions of modular forms (resp. cusp forms) in $M_{k}(\Gamma_0(N),\chi)$ (resp. $S_{k}(\Gamma_0(N),\chi)$) with integer coefficients modulo $m$.

At times, for convenience, we will use the notation $a\equiv_m b$ instead of $a\equiv b\ (\mathrm{mod}\ m)$.

~

The construction of modular forms requires the utilization of the following theorem \cite[Theorem 3]{gordon1993multiplicative}:

\begin{theorem}[B. Gordon, K. Hughes]
\label{eta-quotient}
Let
$$f(z)=\prod_{\delta|N}\eta^{r_\delta}(\delta z)$$

\noindent
be a $\eta$-quotient provided

~

\noindent
$\mathrm{(\romannumeral1)}$ $$\sum_{\delta|N}\delta r_\delta\equiv0\ (\mathrm{mod}\ 24);$$
$\mathrm{(\romannumeral2)}$ $$\sum_{\delta|N}\frac{Nr_\delta}{\delta}\equiv0\ (\mathrm{mod}\ 24);$$
$\mathrm{(\romannumeral3)}$ $$k:=\frac{1}{2}\sum_{\delta|N}r_\delta\in\mathbb{Z},$$

\noindent
then

$$f\left(\frac{az+b}{cz+d}\right)=\chi(d)(cz+d)^kf(z),$$

\noindent
for each $\begin{pmatrix}
 a & b\\
 c & d
\end{pmatrix}\in\Gamma_0(N)$ and $\chi$ is a Dirichlet character $(\mathrm{mod}\ N)$ defined by
$$\chi(n):=\left( \frac{(-1)^k\prod_{\delta|N}\delta^{r_\delta}}{n} \right),\ if\ n>0\ and\ (n,6)=1.$$
\end{theorem}

~

If $f(z)$ is holomorphic (resp. vanishes) at all cusps of $\Gamma_0(N)$, then $f(z)\in M_k(\Gamma_0(N)$ $,\chi)$ (resp. $S_k(\Gamma_0(N),\chi)$), as $\eta(z)$ never vanishes on $\mathcal{H}$. The following theorem (cf. \cite{martin1996multiplicative}) provides a useful criterion for computing the orders of an $\eta$-quotient at all cusps of $\Gamma_0(N)$.

\begin{theorem}[Y. Martin]

\label{order of cusp}
Let $c$, $d$, and $N$ be positive integers with $d\mid N$ and $(c,d)=1$. If $f(z)$ is an $\eta$-quotient that satisfies the conditions of Theorem \ref{eta-quotient}, then the order of vanishing of $f(z)$ at the cusp $c/d$ is
$$\frac{N}{24}\sum_{\delta|N}\frac{r_\delta(d^2,\delta^2)}{\delta(d^2,N)}.$$

\end{theorem}

\section{Ramanujan-type congruences}

\noindent
In this section, we will establish the proof of Theorem \ref{infinitely many Ramanujan-type congruences} using the theory of modular forms. However, it's important to note that the generating function of the regular partition function is not a modular form. Nonetheless, for primes $m\geq5$, it turns out that through the careful selection of a function $h_m(n)$, we have
\begin{equation}
    \notag
    \sum_{n=0}^\infty b_k(h_m(n))q^n,
\end{equation}

\noindent
which, in fact, represents the Fourier expansion of a modular form modulo $m$. In fact, we have

~

\begin{theorem}
\label{cusp form1}
Let $m\geq3$ be a prime, then
    \begin{equation}
        \abovedisplayskip=1em
        \notag
        \sum_{n=0}^\infty b_4\left( \frac{mn-1}{8} \right)q^n\in M_{3m-3}(\Gamma_0(256))_m.
        \belowdisplayskip=1em
    \end{equation}
\end{theorem}

\noindent
The case $m=2$ will be discussed in Section $\ref{more example}$.

~

\begin{theorem}
\label{cusp form2}
Let $m\geq7$ be a prime, then
\begin{equation}
    \abovedisplayskip=1em
    \notag
    \sum_{n=0}^\infty b_6\left( \frac{mn-5}{24} \right)q^n\in S_{2m-2}(\Gamma_0(3456),\chi_6)_m,
    \belowdisplayskip=1em
\end{equation}

\noindent
where $\chi_{6}(n)=\left( \frac{6}{n} \right)$. In addition, 
\begin{equation}
    \abovedisplayskip=1em
    \notag
    \sum_{n=0}^\infty b_6\left( \frac{25n-5}{24} \right)q^n\in S_{48}(\Gamma_0(3456),\chi_6)_5.
    \belowdisplayskip=1em
\end{equation}
\end{theorem}

~

\begin{proof}[Proof of Theorem \ref{cusp form1}]

We begin with an $\eta$-quotient
\begin{equation}
    \abovedisplayskip=1em
    \notag
    f(m;z):=\frac{\eta(4z)}{\eta(z)}\frac{\eta^a(4mz)}{\eta^a(mz)}\eta^6(2mz),
    \belowdisplayskip=1em
\end{equation}

\noindent
where $m':=(m\ \mathrm{mod}\ 8)$ and $a:=4-m'$.

It is straightforward to verify that $f(m;z)\equiv_m\eta^{am+1}(4z)\eta^{6m}(2z)\eta^{-am-1}(z)$ satisfies the conditions of Theorem \ref{eta-quotient}. Furthermore, by applying Theorem \ref{order of cusp}, one can compute that it attains the minimal order of vanishing of $((8-m')m+1)/8$ at the cusp $\infty$, $(mm'-1)/8$ at the cusp $0$, and $m/2$ at the cusp $1/2$.

The order is a half-integer due to the potential presence of an irregular cusp at $1/2$. Indeed, an irregular cusp of $\Gamma_0(4)$ exclusively occurs at the cusp $1/2$ with nontrivial nebentypus (see \cite[Chapter 3]{diamond2005first}).

To sum up, we have
\begin{equation}
    \abovedisplayskip=1em
    \notag
    \frac{\eta^{am+1}(4z)}{\eta^{am+1}(z)}\eta^{6m}(2z)\in S_{3m}(\Gamma_0(4),\chi_4), \text{where}\ \chi_4(n)=\bracket{\frac{-1}{n}}.
    \belowdisplayskip=1em
\end{equation}

\noindent
On the other hand,
\begin{equation}
    \abovedisplayskip=1em
    \notag
    f(m;z)=\sum_{n=0}^\infty b_4(n)q^{n+\frac{m(a+4)+1}{8}}\cdot\prod_{n=1}^\infty\frac{(1-q^{4mn})^a}{(1-q^{mn})^a}(1-q^{2mn})^6.
    \belowdisplayskip=1em
\end{equation}

\noindent
Thus,
\begin{equation}
    \abovedisplayskip=1em
    \label{after U(m)}
    \begin{aligned}
    &\ \ \ \ \frac{\eta^{am+1}(4z)}{\eta^{am+1}(z)}\eta^{6m}(2z)\ |\ U(m)\\
    &\equiv_m\left(\sum_{n=0}^\infty b_4(n)q^{n+\frac{m(a+4)+1}{8}}\ |\ U(m)\right)\cdot\prod_{n=1}^\infty\frac{(1-q^{4n})^a}{(1-q^{n})^a}(1-q^{2n})^6.
    \end{aligned}
    \belowdisplayskip=1em
\end{equation}

~

\noindent
As for the right-hand side of (\ref{after U(m)}),
\begin{equation}
\notag
\sum_{n=0}^\infty b_4(n)q^{n+\frac{m(a+4)+1}{8}}\ |\ U(m)={\sum_{n\geq0}}^* b_4(n)q^{\frac{8n+m(a+4)+1}{8m}},
\end{equation}

\noindent
where ${\sum}^*$ indicates taking integral power coefficients of $q$, i.e.,
$$8n+m(a+4)+1\equiv0\ (\mathrm{mod}\ 8m).$$

\noindent
Verifying the condition is straightforward and reveals its equivalence to $m\mid 8n+1$.

As for the left-hand side of (\ref{after U(m)}), we have
$$\frac{\eta^{am+1}(4z)}{\eta^{am+1}(z)}\eta^{6m}(2z)\ |\ U(m)\equiv_m\frac{\eta^{am+1}(4z)}{\eta^{am+1}(z)}\eta^{6m}(2z)\ |\ T(m),$$

\noindent
where $T(m)$ denotes the usual Hecke operator acting on $S_{3m}(\Gamma_0(4),\chi_4)$.

We now analyze the $\eta$-product $\eta^4(z)\eta^2(2z)\eta^4(4z)$. According to Theorem \ref{eta-quotient} and Theorem \ref{order of cusp}, it qualifies as a cusp form of weight $5$ and level $4$ with nebentypus $\chi_4$. It possesses the minimal order of vanishing of $1/2$ at the cusp $1/2$ and $1$ at the other cusps. Since $\eta(z)$ never vanishes on $\mathcal{H}$, we can write
\begin{equation}
    \abovedisplayskip=1em
    \notag
    \frac{\eta^{am+1}(4z)}{\eta^{am+1}(z)}\eta^{6m}(2z)\ |\ T(m)=\eta^4(z)\eta^2(2z)\eta^4(4z)g(m;z),
    \belowdisplayskip=1em
\end{equation}

\noindent
where $g(m;z)\in M_{3m-5}(\Gamma_0(4))$.

To summarize, we have established
\begin{equation}
    \abovedisplayskip=1em
    \label{b_4(n)}
    \sum_{\genfrac{}{}{0pt}{}{n\geq0}{m|8n+1}} b_4(n)q^{\frac{8n+m(a+4)+1}{8m}}\equiv_m\eta^4(z)\eta^2(2z)\eta^4(4z)g(m;z)\cdot\prod_{n=1}^\infty\frac{(1-q^{n})^a}{(1-q^{2n})^6(1-q^{4n})^a}.
    \belowdisplayskip=1em
\end{equation}

\noindent
By substituting $q$ with $q^{8}$ and then multiplying both sides of (\ref{b_4(n)}) by $q^{-(a+4)}$, we obtain
$$\sum_{\genfrac{}{}{0pt}{}{n\geq0}{m|8n+1}} b_4(n)q^{\frac{8n+1}{m}}\equiv_m\frac{\eta^{4+a}(8z)\eta^{4-a}(32z)}{\eta^{4}(16z)}g(m;8z),$$

\noindent
which can be written as
\begin{equation}
    \abovedisplayskip=1em
    \label{b_4(n)v2}
    \sum_{n=0}^\infty b_4\left(\frac{mn-1}{8}\right)q^n\equiv_m\frac{\eta^{4+a}(8z)\eta^{4-a}(32z)}{\eta^{4}(16z)} g(m;8z).
    \belowdisplayskip=1em
\end{equation}

\noindent
Employing Theorem \ref{eta-quotient} and Theorem \ref{order of cusp} again, one can confirm that the first term on the right-hand side of \eqref{b_4(n)v2} lies within $M_2(\Gamma_0(256))$. In fact, it possesses the minimal order of vanishing of $8-m'$ at the cusps $c/d$ for $d=1,2,4,8$; $m'$ for $d=32,64,128,256$; and $0$ for $d=16$.

Therefore we obtain
\begin{equation}
    \abovedisplayskip=1em
    \notag
    \sum_{n=0}^\infty b_4\left( \frac{mn-1}{8} \right)q^n\in M_{3m-3}(\Gamma_0(256))_m.
    \belowdisplayskip=1em
\end{equation}

\end{proof}

\begin{proof}[Proof of Theorem \ref{cusp form2}]

For a fixed prime $m\geq7$, let
\begin{equation}
    \abovedisplayskip=1em
    \notag
    f(m;z):=\frac{\eta(6z)}{\eta(z)}\eta^a(mz)\eta^b(2mz)\eta^c(3mz)\eta^d(6mz),
    \belowdisplayskip=1em
\end{equation}

\noindent
where $m':=(m\ \mathrm{mod}\ 24)$ and $a:=(m'\ \mathrm{mod}\ 5)-1$, $b:=\floor{m'/5}-1$, $c:=3-\floor{m'/5}$, $d:=3-(m'\ \mathrm{mod}\ 5)$. It is easy to show that
\begin{equation}
    \abovedisplayskip=1em
    \notag
    f(m;z)\equiv_m\eta^{am-1}(z)\eta^{bm}(2z)\eta^{cm}(3z)\eta^{dm+1}(6z)\in S_{2m}(\Gamma_0(6)).
    \belowdisplayskip=1em
\end{equation}

\noindent
In fact, the order of vanishing at the cusp $s/t$ are
\begin{equation}
    \abovedisplayskip=1em
    \notag
    \begin{cases}
        \frac{m(5m\ \mathrm{mod}\ 24)-5}{24}  & \text{ if } t=1, \\
        \frac{mm'-1}{24} & \text{ if } t=2, \\
        \frac{m(24-m')+1}{24} & \text{ if } t=3, \\
        \frac{m(24-(5m\ \mathrm{mod}\ 24))+5}{24}  & \text{ if } t=6.
    \end{cases}
    \belowdisplayskip=1em
\end{equation}

\noindent
On the other hand,
\begin{equation}
    \abovedisplayskip=1em
    \notag
    f(m;z)=\sum_{n=0}^\infty b_6(n)q^{\frac{24n+m(a+2b+3c+6d)+5}{24}}\cdot\prod_{n=1}^\infty(1-q^{mn})^a(1-q^{2mn})^b(1-q^{3mn})^c(1-q^{6mn})^d.
    \belowdisplayskip=1em
\end{equation}

\noindent
By applying the $U(m)$ operator to $f(z)$ and recognizing that $\modulo{U(m)}{T(m)}{m}$, we arrive at
\begin{equation}
    \abovedisplayskip=1em
    \label{after U/T}
    \begin{aligned}
        &\ \ \ \ {\sum_{n=0}^\infty b_6(n)q^{\frac{24n+m(a+2b+3c+6d)+5}{24}}\ |\ U(m)}\\
        &\equiv_m{\frac{\eta^{am-1}(z)\eta^{bm}(2z)\eta^{cm}(3z)\eta^{dm+1}(6z)\ |\ T(m)}{\prod_{n=1}^\infty(1-q^{n})^a(1-q^{2n})^b(1-q^{3n})^c(1-q^{6n})^d}},
    \end{aligned}
    \belowdisplayskip=1em
\end{equation}

\noindent
where $T(m)$ denotes usual Hecke operator acting on $S_{2m}(\Gamma_0(6))$. As for the left-hand side of (\ref{after U/T}), we have
\begin{equation}
    \abovedisplayskip=1em
    \notag
    \sum_{n=0}^\infty b_6(n)q^{\frac{24n+m(a+2b+3c+6d)+5}{24}}\ |\ U(m)=\sum_{\genfrac{}{}{0pt}{}{n=0}{m|24n+5}}^\infty b_6(n)q^{\frac{24n+m(a+2b+3c+6d)+5}{24m}}.
    \belowdisplayskip=1em
\end{equation}

\noindent
By utilizing Theorem \ref{eta-quotient} and Theorem \ref{order of cusp}, one can confirm that $\eta^2(z)\eta^2(2z)\eta^2(3z)$ $\eta^2(6z)$ $\in S_4(\Gamma_0(6))$ and holds an order of $1$ at all cusps. Consequently, we can write
$$\eta^{am-1}(z)\eta^{bm}(2z)\eta^{cm}(3z)\eta^{dm+1}(6z)\ |\ T(m)=\eta^2(z)\eta^2(2z)\eta^2(3z)\eta^2(6z)g(m;z),$$ 

\noindent
where $g(m;z)\in M_{2m-4}(\Gamma_0(6))$. Hence
\begin{equation}
    \abovedisplayskip=1em
    \notag
    \modulo{\sum_{\genfrac{}{}{0pt}{}{n=0}{m|24n+5}}^\infty b_6(n)q^{\frac{24n+5}{24m}}}{\eta^{2-a}(z)\eta^{2-b}(2z)\eta^{2-c}(3z)\eta^{2-d}(6z)g(m;z)}{m}.
    \belowdisplayskip=1em
\end{equation}

\noindent
Substituting $q$ with $q^{24}$ reveals that
\begin{equation}
    \abovedisplayskip=1em
    \notag
    \modulo{\sum_{\genfrac{}{}{0pt}{}{n=0}{m|24n+5}}^\infty b_6(n)q^{\frac{24n+5}{m}}}{\eta^{2-a}(24z)\eta^{2-b}(48z)\eta^{2-c}(72z)\eta^{2-d}(144z)g(m;24z)}{m}.
    \belowdisplayskip=1em
\end{equation}

\noindent
As $b_6(n)$ vanishes for non-integer $n$, thus
\begin{equation}
    \abovedisplayskip=1em
    \notag
    \modulo{\sum_{n=0}^\infty b_6\bracket{\frac{mn-5}{24}}q^{n}}{\eta^{2-a}(24z)\eta^{2-b}(48z)\eta^{2-c}(72z)\eta^{2-d}(144z)g(m;24z)}{m}.
    \belowdisplayskip=1em
\end{equation}

\noindent
Furthermore, it can be verified that
$$\eta^{2-a}(24z)\eta^{2-b}(48z)\eta^{2-c}(72z)\eta^{2-d}(144z)\in S_2(\Gamma_0(3456),\chi_6).$$

\noindent
In fact, the order of vanishing at the cusp $s/t$ is
\begin{equation}
    \abovedisplayskip=1em
    \notag
    \begin{cases}
        24-(5m\ \text{mod}\ 24) & \text{ if } t=1,2,3,4,6,8,12,24, \\
        m' & \text{ if } t=9,18,27,36,54,72,108,216, \\
        24-m' & \text{ if } t=16,32,48,64,96,128,192,384, \\
        (5m\ \text{mod}\ 24)  & \text{ if } t=144,288,432,576,864,1152,1728,3456.
    \end{cases}
    \belowdisplayskip=1em
\end{equation}

\noindent
Combining this with $g(m;24z)\in M_{2m-4}(\Gamma_0(144))$, we obtain
\begin{equation}
    \abovedisplayskip=1em
    \notag
    \sum_{n=0}^\infty b_6\bracket{\frac{mn-5}{24}}q^{n}\in S_{2m-2}(\Gamma_0(3456),\chi_{6})_m.
    \belowdisplayskip=1em
\end{equation}

\begin{remark}
    The proof is not suitable for $m=5$ because
    $$f(m;z)\equiv_m\eta^{am-1}(z)\eta^{bm}(2z)\eta^{cm}(3z)\eta^{dm+1}(6z)\notin S_{2m}(\Gamma_0(6)).$$
\end{remark}

~

\noindent
Now we turn to the case $m=5$. Let
\begin{equation}
    \abovedisplayskip=1em
    \notag
    f(z):=\frac{\eta(6z)}{\eta(z)}\frac{\eta^3(75z)\eta^2(150z)}{\eta(50z)}.
    \belowdisplayskip=1em
\end{equation}

\noindent
Note that $f(z)\equiv_5\eta^{-1}(z)\eta^{-25}(2z)\eta^{75}(3z)\eta^{51}(6z)\in S_{50}(\Gamma_0(6))$. On the other hand,
\begin{equation}
    \abovedisplayskip=1em
    \notag
    f(z)=\sum_{n=0}^\infty b_6(n)q^{n+20}\cdot\prod_{n=1}^\infty(1-q^{50n})^{-1}(1-q^{75n})^3(1-q^{150n})^2.
    \belowdisplayskip=1em
\end{equation}

\noindent
Now we use the fact that $U(25)\equiv_5T(25)$ to obtain that
\begin{equation}
    \abovedisplayskip=1em
    {\sum_{n=0}^\infty b_6(n)q^{{n+20}{}}\ |\ U(25)}\equiv_5{\frac{\eta^{-1}(z)\eta^{-25}(2z)\eta^{75}(3z)\eta^{51}(6z)\ |\ T(25)}{\prod_{n=1}^\infty(1-q^{2n})^{-1}(1-q^{3n})^3(1-q^{6n})^2}},
    \belowdisplayskip=1em
\end{equation}

\noindent
Hence
\begin{equation}
    \abovedisplayskip=1em
    \notag
    {\sum_{\genfrac{}{}{0pt}{}{n=0}{\modulo{n}{5}{25}}}^\infty b_6(n)q^{\frac{n+20}{25}}}\equiv_5{\frac{\eta^2(z)\eta^2(2z)\eta^2(3z)\eta^2(6z)g(z)}{\prod_{n=1}^\infty(1-q^{2n})^{-1}(1-q^{3n})^3(1-q^{6n})^2}},
    \belowdisplayskip=1em
\end{equation}

\noindent
where $g(z)\in M_{46}(\Gamma_0(6))$. Substituting $q$ with $q^{24}$, obtaining
\begin{equation}
    \abovedisplayskip=1em
    \notag
    {\sum_{\genfrac{}{}{0pt}{}{n=0}{\modulo{n}{5}{25}}}^\infty b_6(n)q^{\frac{24n+5}{25}}}\equiv_5{{\eta^2(24z)\eta^3(48z)\eta^{-1}(72z)g(24z)}},
    \belowdisplayskip=1em
\end{equation}

\noindent
where $\eta^2(24z)\eta^3(48z)\eta^{-1}(72z)\in S_2(\Gamma_0(3456),\chi_6)$ and $g(24z)\in M_{46}(\Gamma_0(144))$. Finally we obtain
\begin{equation}
    \abovedisplayskip=1em
    \notag
    \sum_{n=0}^\infty b_6\bracket{\frac{25n-5}{24}}q^n\in S_{48}(\Gamma_0(3456),\chi_6)_5.
    \belowdisplayskip=1em
\end{equation}

\end{proof}

\noindent
We require a crucial result by Serre (see \cite[(6.4)]{serre1974divisibilite}, \cite[Lemma 2.63]{ono2004web}), which play a pivotal role in establishing the existence of Ramanujan-type congruences.

\begin{theorem}[J.-P. Serre]
\label{Serre's theorem}
The set of primes $l\equiv-1\ (\mathrm{mod}\ Nm)$ such that
$$f\ |\ T(l)\equiv0\ (\mathrm{mod}\ m)$$

\noindent
for each $f(z)\in M_k(\Gamma_0(N),\psi)_m$ has positive density, where $T(l)$ denotes the usual Hecke operator acting on $M_k(\Gamma_0(N),\psi)$.
\end{theorem}

\noindent
Now, Theorem \ref{infinitely many Ramanujan-type congruences} follows immediately as a corollary of the next two theorems.

\begin{theorem}
\label{main theorem}
Let $m\geq3$ be a prime. A positive density of the primes $l$ possess the property that
\begin{equation}
    \abovedisplayskip=1em
    \notag
    b_4\left( \frac{mln-1}{8} \right)\equiv0\ (\mathrm{mod}\ m)
    \belowdisplayskip=1em
\end{equation}

\noindent
for every nonnegative integer $n$ coprime to $l$.
\end{theorem}

\begin{theorem}
\label{main theorem2}
Let $m\geq7$ be a prime. A positive density of the primes $l$ possess the property that
\begin{equation}
    \abovedisplayskip=1em
    \notag
    b_6\left( \frac{mln-5}{24} \right)\equiv0\ (\mathrm{mod}\ m)
    \belowdisplayskip=1em
\end{equation}

\noindent
for every nonnegative integer $n$ coprime to $l$. In addition, A positive density of the primes $l$ possess the property that
\begin{equation}
    \abovedisplayskip=1em
    \notag
    b_6\left( \frac{25ln-5}{24} \right)\equiv0\ (\mathrm{mod}\ 5)
    \belowdisplayskip=1em
\end{equation}

\noindent
for every nonnegative integer $n$ coprime to $l$.
\end{theorem}

~

\begin{proof}[Proof of Theorem \ref{main theorem}]Let
\begin{equation}
    \abovedisplayskip=1em
    \notag
    F(m;z)=\sum_{n=0}^\infty b_4\left( \frac{mn-1}{8} \right)q^n,
    \belowdisplayskip=1em
\end{equation}

\noindent
then $F(m;z)\in M_{3m-3}(\Gamma_0(256))_m$.

For a fixed prime $m\geq5$, let $S(m)$ denote the set of primes $l$ such that
$$f\ |\ T(l)\equiv0\ (\mathrm{mod}\ m)$$

\noindent
for every $f\in M_{3m-3}(\Gamma_0(256))$. According to Theorem \ref{Serre's theorem}, the set $S(m)$ includes a positive density of primes. Therefore, for $l\in S(m)$, we have
$$F(m;z)\ |\ T(l)\equiv0\ (\mathrm{mod}\ m).$$

\noindent
Subsequently, employing the theory of Hecke operators, we obtain
\begin{equation}
    \abovedisplayskip=1em
    \notag
    F(m;z)\ |\ T(l)=\sum_{n=0}^\infty\left( b_4\left( \frac{mln-1}{8} \right)+l^{3m-4}b_4\left( \frac{mn/l-1}{8} \right) \right)q^n\equiv0\ (\mathrm{mod}\ m).
    \belowdisplayskip=1em
\end{equation}

\noindent
Since $b_4(n)$ vanishes when $n$ is not an integer, we have
$$b_4\left(\frac{mn/l-1}{8}\right)=0$$

\noindent
for each $n$ coprime to $l$. Consequently,
$$b_4\left(\frac{mln-1}{8}\right)\equiv 0\ (\mathrm{mod}\ m)$$

\noindent
holds for every integer $n$ coprime to $l$. Moreover, the set of such primes $l$ possesses a positive density.

\end{proof}

\begin{proof}[Proof of Theorem \ref{main theorem2}]

Let
\begin{equation}
    \abovedisplayskip=1em
    \notag
    F(m;z)=\sum_{n=0}^\infty b_6\left( \frac{mn-5}{24} \right)q^n\in S_{2m-2}(\Gamma_0(3456),\chi_6)_m.
    \belowdisplayskip=1em
\end{equation}

\noindent
According to Theorem \ref{Serre's theorem}, the set of primes $l$ for which
$$ F(m;z)\ |\ T(l)\equiv0\ (\mathrm{mod}\ m)$$

\noindent
has positive density. Here, $T(l)$ denotes the Hecke operator acting on $S_{2m-2}(\Gamma_0(3456),$ $\chi_6)$. Furthermore, applying the theory of Hecke operators, we obtain
\begin{equation}
    \abovedisplayskip=1em
    \notag
    \sum_{n=0}^\infty F(m;z)\ |\ T(l)=\sum_{n=0}^\infty \bracket{b_6\left( \frac{mln-5}{24} \right)+\bracket{\frac{6}{l}}l^{2m-3}b_6\left( \frac{mn/l-5}{24} \right)}q^n.
    \belowdisplayskip=1em
\end{equation}

\noindent
Since $b_6(n)$ vanishes for non-integer $n$, we have
$$b_6\left(\frac{mn/l-5}{24}\right)=0$$

\noindent
when $(n,l)=1$. As a result,
\begin{equation}
    \abovedisplayskip=1em
    \notag
    \modulo{b_6\bracket{\frac{mln-5}{24}}}{0}{m}
    \belowdisplayskip=1em
\end{equation}

\noindent
holds for every integer $n$ with $(n,l)=1$. Moreover, the set of such primes $l$ possesses a positive density.

The proof for $m=5$ is similar.

\end{proof}

\noindent
Considering that the number of choices for $l$ is infinite, let us select $l>3$. After replacing $n$ with $12nl+ml+12$, we observe that $b_4(ml^2n+ml+(m^2l^2-1)/12)\equiv0\ (\mathrm{mod}\ m)$ holds for every nonnegative integer $n$. A similar approach can be employed for $b_6(n)$. Consequently, we establish Theorem \ref{infinitely many Ramanujan-type congruences}. Moreover, since the options for $l$ are unlimited, in conjunction with the Chinese Remainder Theorem and previous results, we derive

~

\begin{corollary}
If $m$ is a squarefree integer, then there exist infinitely many Ramanujan-type congruences of $b_4(n)$ modulo $m$. Similarly, if $k$ is a squarefree integer coprime to $10$, then there exist infinitely many Ramanujan-type congruences of $b_6(n)$ modulo $k$.
\end{corollary}

While we can derive results about the distribution of nonzero residues similar to the previous paper, it is important to note that this is not the main focus of our current work.

~

\section{Examples of Ramanujan-type congruences}

\label{more example}

Here, we introduce a theorem by Sturm \cite[Theorem 1]{sturm1987congruence}, which offers a useful criterion for determining when modular forms with integer coefficients become congruent to zero modulo a prime through finite computation.

\begin{theorem}[J. Sturm]
\label{Sturm's theorem}
Suppose $f(z)=\sum_{n=0}^\infty a(n)q^n\in M_k(\Gamma_0(N),\chi)_m$ such that
$$a(n)\equiv0\ (\mathrm{mod}\ m)$$

\noindent
for all $n\leq \frac{kN}{12}\prod_{p|N}\left( 1+\frac1p \right)$. Then $a(n)\equiv0\ (\mathrm{mod}\ m)$ for all $n\in\mathbb{Z}$.
\end{theorem}

~

\label{examples}

\noindent
According to Theorem \ref{Sturm's theorem}, we discover that
$$\sum_{n=0}^\infty b_4\left(\frac{mn-1}{8}\right)q^n\ |\ T(l)\equiv0\ (\mathrm{mod}\ m)$$

\noindent
is satisfied for the following combinations of $m$ and $l$:
\begin{enumerate}
    \item $m=5$ and $l=809,839,1249,1279,1319,1489,1811$;
    \item $m=7$ and $l=1889,1901$.
\end{enumerate}

\noindent
A straightforward computation shows that

~

\begin{proposition}
For the given values of $m$ and $l$, the congruence
\begin{equation}
    \abovedisplayskip=1em
    \notag
    b_4\bracket{ml(ln+j)+\frac{m^2l^2-1}{8}}\equiv0\ (\mathrm{mod}\ m)
    \belowdisplayskip=1em
\end{equation}

\noindent
satisfied for each $n$ and $1\leq j\leq l-1$.
\end{proposition}

~

\noindent
In fact, we have a full characterization of $b_4(n)$ modulo $2$. Since
\begin{equation}
    \abovedisplayskip=1em
    \notag
    \sum_{n=0}^\infty b_4\bracket{\frac{n-1}{8}}q^n=\sum_{n=0}^\infty b_4(n)q^{8n+1}=\frac{\eta(32z)}{\eta(8z)}\equiv_2\eta^{24}(z)=\sum_{n=1}^\infty\tau(n)q^n,
    \belowdisplayskip=1em
\end{equation}

\noindent
where $\tau(n)$ is the Ramanujan tau function and $\tau(n)$ is odd if and only if $n$ is an odd square (see for example \cite[Proof of Theorem 2.2.1]{berndt2006number}). Thus $b_4(n)$ is odd if and only if $8n+1$ is an odd square, or equivalently, $b_4(n)$ is odd if and only if $n$ is a triangular number. (Full characterization of $b_4(n)$ modulo $2$ is already proved by Cherubini and Mercuri \cite[Theorem 1.3]{cherubini2022parity}, but the argument here is simpler.) So we have

~

\begin{proposition}
For each odd prime $m$, the congruence
\begin{equation}
    \abovedisplayskip=1em
    \notag
    b_4(m^2n+j)\equiv0\ (\mathrm{mod}\ 2)
    \belowdisplayskip=1em
\end{equation}

\noindent
satisfied for each $n$, where $j$ satisfies $m\,||\,8j+1$.
\end{proposition}

\noindent
Now we are going to prove Theorem \ref{congruences mod 3}.

\begin{proof}[Proof of Theorem \ref{congruences mod 3}]
    We first observe that $\modulo{l}{13,17,19,23}{24}$ is equivalent to $\left(\frac{-6}{l}\right) = -1$.

    By Theorem 3.1 of \cite{andrews2010arithmetic}, we have
    \begin{equation}
        \abovedisplayskip=1em
        \notag
        \begin{aligned}
            \sum_{n=0}^\infty b_4(3n)q^n&=\frac{(q^4;q^4)_\infty(q^6;q^6)_\infty^4}{(q;q)_\infty^3(q^{12};q^{12})^2_\infty}\\
            &\equiv_3\frac{(q^2;q^2)^{12}_\infty}{(q;q)_\infty^3(q^4;q^4)_\infty^5}\\
            &=\frac{(q^2;q^2)^{13}_\infty}{(q;q)_\infty^5(q^4;q^4)_\infty^5}\cdot\frac{(q;q)^2_\infty}{(q^2;q^2)_\infty}.
        \end{aligned}
        \belowdisplayskip=1em
    \end{equation}

    \noindent
    By the fourth identity of Theorem 1.1 (2) of \cite{oliver2013eta}, we have
    \begin{equation}
        \abovedisplayskip=1em
        \notag
        \frac{(q^2;q^2)^{13}_\infty}{(q;q)_\infty^5(q^4;q^4)_\infty^5}=\sum_{m=1}^\infty\bracket{\frac{-6}{m}}mq^{(m^2-1)/24}.
        \belowdisplayskip=1em
    \end{equation}

    \noindent
    By the first identity of Theorem 1.2 of \cite{oliver2013eta}, we have
    \begin{equation}
        \abovedisplayskip=1em
        \notag
        \frac{(q;q)^2_\infty}{(q^2;q^2)_\infty}=\sum_{k=-\infty}^\infty(-1)^kq^{k^2}.
        \belowdisplayskip=1em
    \end{equation}

    \noindent
    Thus
    \begin{equation}
        \abovedisplayskip=1em
        \label{b_4(3n)}
        \sum_{n=0}^\infty b_4(3n)q^n\equiv_3\sum_{m=1}^\infty\sum_{k=-\infty}^\infty(-1)^k\bracket{\frac{-6}{m}}mq^{k^2+(m^2-1)/24}.
        \belowdisplayskip=1em
    \end{equation}

    \noindent
    Let $n_0=ln+(l^2-1)/24$ with $(n,l)=1$. Now, we analyze the coefficient of $q^{n_0}$ on both sides of \eqref{b_4(3n)}. As for the left-hand side, the coefficient is
    \begin{equation}
        \abovedisplayskip=1em
        \notag
        b_4\left(3ln+\frac{l^2-1}{8}\right).
        \belowdisplayskip=1em
    \end{equation}

    \noindent
    As for the right-hand side, we solve $n_0=k^2+(m^2-1)/24$ for $n$, which is equivalent to $24ln+l^2=24k^2+m^2$. First, assuming that $l\mid k$, we would have $l\mid m$, and consequently, $l^2\mid 24ln$, which contradicts the condition $(n,l)=1$. Now, let us assume that $l\nmid k$, we have
    \begin{equation}
        \abovedisplayskip=1em
        \notag
        \modulo{-24k^2}{m^2}{l},
        \belowdisplayskip=1em
    \end{equation}

    \noindent
    i.e.
    \begin{equation}
        \abovedisplayskip=1em
        \notag
        \modulo{-6}{\bracket{\frac{m}{2k}}^2}{l},
        \belowdisplayskip=1em
    \end{equation}

    \noindent
    which contradicts $\bracket{\frac{-6}{l}}=-1$. Therefore, there is no such $n$, and the coefficient of $q^{n_0}$ on the right-hand side of \eqref{b_4(3n)} is $0$. Consequently, we obtain
    \begin{equation}
        \abovedisplayskip=1em
        \notag
        \modulo{b_4\left(3ln+\frac{l^2-1}{8}\right)}03
        \belowdisplayskip=1em
    \end{equation}

    \noindent
    for each $n$ coprime to $l$ and $\bracket{\frac{-6}{l}}=-1$.

    Again by Theorem 3.1 of \cite{andrews2010arithmetic}, we have
    \begin{equation}
        \abovedisplayskip=1em
        \notag
        \sum_{n=0}^\infty (-1)^nb_4(3n+1)q^n=\frac{\phi(q^3)\psi(q^3)}{\phi^2(q)},
        \belowdisplayskip=1em
    \end{equation}

    \noindent
    where
    \begin{equation}
        \abovedisplayskip=1em
        \notag
        \phi(q)=\sum_{n=-\infty}^\infty q^{n^2}\quad\text{and}\quad\psi(q)=\sum_{n=0}^\infty q^{n(n+1)/2}.
        \belowdisplayskip=1em
    \end{equation}

    \noindent
    Note that
    \begin{equation}
        \abovedisplayskip=1em
        \notag
        {\frac{\phi(q^3)\psi(q^3)}{\phi^2(q)}}\equiv{\phi(q)\psi(q^3)}\pmod3.
        \belowdisplayskip=1em
    \end{equation}

    \noindent
    Thus
    \begin{equation}
        \abovedisplayskip=1em
        \label{phi psi^3}
        \sum_{n=0}^\infty (-1)^nb_4(3n+1)q^n\equiv_3\phi(q)\psi(q^3)=\sum_{m=1}^\infty\sum_{k=-\infty}^\infty q^{m^2+3k(k+1)/2}.
        \belowdisplayskip=1em
    \end{equation}

    \noindent
    Let $n_0=ln+(3l^2-3)/8$ with $(n,l)=1$. Now, we analyze the coefficient of $q^{n_0}$ on both sides of \eqref{phi psi^3}. As for the left-hand side, the coefficient is
    \begin{equation}
        \abovedisplayskip=1em
        \notag
        (-1)^{n_0}b_4\left(3ln+\frac{9l^2-1}{8}\right).
        \belowdisplayskip=1em
    \end{equation}

    \noindent
    As for the right-hand side, we solve $n_0=3k(k+1)/2+m^2$ for $n$, which is equivalent to $8ln+3l^2=3(2k+1)^2+8m^2$. First, assuming that $l\mid 2k+1$, we would have $l\mid m$, and consequently, $l^2\mid 8ln$, which contradicts the condition $(n,l)=1$. Now, let us assume that $l\nmid 2k+1$, we have
    \begin{equation}
        \abovedisplayskip=1em
        \notag
        \modulo{-6(2k+1)^2}{(4m)^2}{l},
        \belowdisplayskip=1em
    \end{equation}

    \noindent
    i.e.
    \begin{equation}
        \abovedisplayskip=1em
        \notag
        \modulo{-6}{\bracket{\frac{4m}{2k+1}}^2}{l},
        \belowdisplayskip=1em
    \end{equation}

    \noindent
    which contradicts $\bracket{\frac{-6}{l}}=-1$. Therefore, there is no such $n$, and the coefficient of $q^{n_0}$ on the right-hand side of \eqref{phi psi^3} is $0$. Consequently, we obtain
    \begin{equation}
        \abovedisplayskip=1em
        \notag
        \modulo{b_4\left(3ln+\frac{9l^2-1}{8}\right)}03
        \belowdisplayskip=1em
    \end{equation}

    \noindent
    for each $n$ coprime to $l$ and $\bracket{\frac{-6}{l}}=-1$.
\end{proof}

\noindent
We also obtain congruences of $b_6(n)$ modulo $5$.

\begin{proposition}
For $l=1973,2711$, the congruence
\begin{equation}
    \abovedisplayskip=1em
    \notag
    b_6\bracket{25l(ln+j)+\frac{125l^2-5}{24}}\equiv0\ (\mathrm{mod}\ 5)
    \belowdisplayskip=1em
\end{equation}

\noindent
satisfied for each $n$ and $1\leq j\leq l-1$.
\end{proposition}

\begin{proof}
    For $l=1973,2711$, by Sturm's Theorem \ref{Sturm's theorem}, we have
    \begin{equation}
        \abovedisplayskip=1em
        \notag
        \sum_{n=0}^\infty b_6\bracket{\frac{25n-5}{24}}q^n\mid T(l)\equiv 0\ (\text{mod}\ 5).
        \belowdisplayskip=1em
    \end{equation}

    \noindent
    The desired result is now immediate after some elementary computation.
\end{proof}

~

\noindent
For $m\geq7$, we are unable to provide a specific example of $b_6(n)$, although such examples do exist. In cases where $m=7$ or $m=11$, our computations show that for primes $l\leq25000$,
$$\sum_{n=0}^\infty b_6\left(\frac{mn-5}{24}\right)q^n\ |\ T(l)\not\equiv0\ (\mathrm{mod}\ m).$$

~

\section{Open problems}

In this paper, we establish the existence of infinitely many Ramanujan-type congruences modulo $m$ for $b_6(n)$, where $m\geq5$. Additionally, congruences modulo $3$ have been explored (see \cite{ahmed2016new}). However, the question arises: what occurs when $m=2$?

\begin{problem}
For $b_6(n)$, either discover a Ramanujan-type congruence modulo $2$ or prove that no such congruence exists.
\end{problem}

While Ramanujan-type congruences modulo nearly all primes $m$ do exist, it is important to note that discovering them may require extensive computations. We encourage readers who are interested to explore and seek examples of congruences modulo different primes.

\section*{Acknowledgement}

We would like to acknowledge the inspiration that arose from reading the papers of Ono \cite{ono2000distribution} and Lovejoy \cite{lovejoy2001divisibility}, which sparked the ideas behind this work.

\end{document}